\documentclass{amsart}
\usepackage{tikz}

\newtheorem{theorem}{Theorem}[section]
\newtheorem{lemma}[theorem]{Lemma}
\newtheorem{conjecture}[theorem]{Conjecture}

\numberwithin{equation}{section}

\begin{document}

\title{On the Prevalence of Bridge Graphs Among Non-3-Connected Cubic Non-Hamiltonian Graphs}

\author{Rishi Advani}
\email{ra534@cornell.edu}

\date{August 27, 2019}

\keywords{Hamiltonian graph, cubic bridge graph}

\begin{abstract}
There is empirical evidence supporting the claim that almost all cubic non-Hamiltonian graphs are bridge graphs. In this paper, we pose a related conjecture and prove that the original claim holds for non-3-connected graphs if the conjecture is true.
\end{abstract}

\maketitle

\section{Introduction}

Every mention of a graph from here on is referring to a connected undirected graph with no self loops and at most one edge between any pair of vertices.

In 2010, Filar, Haythorpe, and Nguyen conjectured that almost all cubic non-Hamiltonian graphs are bridge graphs~\cite{A}.
In their paper, Filar et al.\ note that all cubic bridge graphs are non-Hamiltonian. It is straightforward to check if a graph is bridge in polynomial time. If the conjecture is true, if we checked a graph for a bridge and found none, we could be confident that the graph was Hamiltonian. In this paper, we will analyze a subset of the graphs considered in the original conjecture, and provide further evidence for it.

\section{Definitions}

A cubic graph is one where every vertex is connected to exactly 3 other vertices. A cycle is a sequence of distinct vertices $v_1, v_2, v_3 \dots v_k$, such that $v_k$ is connected to $v_1$ and for all $i$, $1 \leq i \leq k-1$, $i \in \mathbf{N}$, $v_i$ is connected to $v_{i+1}$. A Hamiltonian cycle is a cycle that contains all the vertices in the graph. A Hamiltonian graph is one that has a Hamiltonian cycle. The Hamiltonian Cycle Problem is the problem of determining whether a given graph has a Hamiltonian cycle. A bridge is an edge whose removal would disconnect the graph. A bridge graph is a graph containing at least one bridge. A biconnected graph is a nonbridge graph in which there exists an edge whose removal would cause the graph to become bridge. A bi-bridge is a set of two edges in a biconnected graph whose removal disconnect the graph. An induced subgraph of a graph is a graph consisting of a specified subset of the nodes and all the edges whose endpoints are both in the subset.

\section{Bridge Construction}

Take any biconnected cubic graph $G$. Find the bi-bridge that most evenly partitions the nodes of the graph into two sets. If there is more than one such bi-bridge, arbitrarily pick one.

Let $G_1$ be the induced subgraph of $G$ that has more distinct edges in cycles. Two edges are distinct if there does not exist a graph automorphism that maps each edge to the other.

We perform the following construction to obtain a unique bridge graph $G'$ (note that it is not yet cubic):

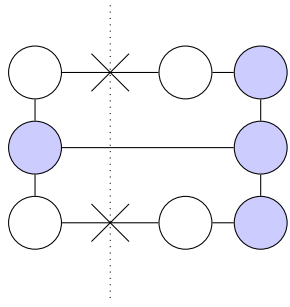
\begin{center}
\begin{tikzpicture}
[every node/.style={circle,draw=black,fill=white,inner sep=0,minimum size=20}]
\tikzstyle{new}=[fill=blue!20]
\node (b1) at (1,3) {};
\node (b2) at (3,3) {};
\node (b3) at (1,1) {};
\node (b4) at (3,1) {};
\node (new1) at (1,2) [new] {};
\node (new2) at (4,2) [new] {};
\node (new3) at (4,3) [new] {};
\node (new4) at (4,1) [new] {};

\foreach \from/\to in {b1/b2,b3/b4,new1/new2,new3/new2,new4/new2,b1/new1,b3/new1,b2/new3,b4/new4}
\draw (\from) -- (\to);

\draw [dotted] (2,0) -- (2,4);

\draw (1.75,3.25) -- (2.25,2.75);
\draw (1.75,2.75) -- (2.25,3.25);

\draw (1.75,1.25) -- (2.25,0.75);
\draw (1.75,0.75) -- (2.25,1.25);
\end{tikzpicture}
\end{center}

Blue nodes denote new nodes added in the construction. `X's denote edges removed during the construction. The dotted line denotes the partition between the two parts of the graph separated by the bi-bridge.

Let $G_1'$ be the induced subgraph of $G'$ created by splitting the graph by removing the bridge, and picking the side that corresponds to $G_1$.

\section{Mapping}

Pick the node in $G_1'$ whose corresponding node in $G'$ is incident to the bridge. Let us denote it as $v_0$. Every node in $G_1'$ is at some distance from $v_0$. Let $d_{max}$ be the maximum distance possible in $G_1'$.
	
\begin{lemma}
	There are at least $d_{max}$ distinct edges in $G_1'$.
\end{lemma}
	
\begin{proof}
	If a node $v_1$ has distance $d$, its neighbors must have distances $d-1$, $d$, or $d+1$. If a neighbor $v_2$ had distance less than $d-1$, this would be a contradiction because $v_1$ could reach $v_0$ faster via $v_2$. If $v_2$ has distance greater than $d+1$, it could reach $v_0$ faster via $v_1$.
	
	Furthermore, any minimal path from $v_1$ to $v_0$ must go through one of its neighbors, so the distance of one of those neighbors must be $d-1$.
	
	If an edge joins a node of distance $d$ and a node of distance $d-1$, we say that edge \emph{facilitates} a distance $d$.
	
	For each $d \in [1,d_{max}]$, there exists at least one node with distance $d$. That node has a neighbor of distance $d-1$, so it is incident to an edge that facilitates a distance $d$. Two edges that facilitate different distances must be distinct. Therefore, there are at least $d_{max}$ distinct edges.
\end{proof}

\begin{lemma}
	Every edge in $G_1'$ is in a cycle.
\end{lemma}

\begin{proof}
	If there exists a bridge in $G_1'$, splitting the graph would either leave the nodes that were incident to the bi-bridge on the same side or on different sides.
	
	If they were on the same side, the same edge would be a bridge in $G$, which is a contradiction since $G$ is biconnected.
	
	If they were on different sides, one of the three new edges added in the construction (excluding the edge that is already known to be a bridge) is a bridge. If we remove any of these, $G'$ is an induced subgraph of the resulting graph. This is a contradiction, since $G'$ is a connected graph. Thus, $G_1'$ is nonbridge, which implies that every edge is in a cycle.
\end{proof}

\begin{theorem}	
	For $n$ sufficiently large, we can generate arbitrarily many unique cubic bridge graphs of size $n+4$ from biconnected cubic graphs of size $n$.
\end{theorem}

\begin{proof}
To maximize the number of nodes in a graph with a certain maximal distance, the nodes must be arranged in a tree. As the number of nodes increase, the height of the tree increases correspondingly. By Lemmas 4.1 and 4.2, there are at least $d_{max}$ distinct edges in cycles in $G_1'$. Thus, the number of distinct edges in cycles in post-construction graphs increases without bound.

Our post-construction graphs still have two nodes with degree 2. To remedy this, for each distinct edge in a cycle, we will remove it and join its endpoints to the aforementioned two nodes. The choice of which node is joined to which endpoint is arbitrary. The resulting graph is cubic and bridge. It is guaranteed to be connected since the edge we removed was part of a cycle. Furthermore, the part of the graph corresponding to $G_1'$ is nonbridge because any cycle that used the removed edge can now use the new nodes from the construction.

For every biconnected cubic graph, we are able to construct cubic bridge graphs that are all unique from each other and unique from analogous bridge graphs constructed from different biconnected graphs.
\end{proof}

\section{Reduction of Size}

The next step is to find a way to generate bridge graphs of size $n$ from the graphs of size $n+4$.

Let $A$ denote the resulting graph after performing a `cycle insertion' on $G_1'$. Let $k$ be the number of distinct edges in cycles in the graph.

Now create an induced subgraph by taking the $\lfloor \sqrt[5]{k} \rfloor$ nodes of $G_1'$ closest to $v_0$ (including $v_0$). Then remove all nodes of maximal distance, and denote the resulting induced subgraph as $A'$.

We have the following cases for the structure of $A'$:

\begin{description}
\item[Isolated Triangles] If there exists a triangle, such that it does not share an edge with another triangle, then we can `reduce' the triangle. We remove the triangle, and replace it with a new node, joining it to all the external neighbors of the vertices of the triangle. The resulting overall graph has $n+2$ nodes, and is still connected, bridge, and cubic. No node's distance has increased after the construction.

\item[Adjacent Triangles] If there exists at least one triangle, but every triangle share an edge with another triangle, we take a pair of adjacent triangles, and remove them from the graph. If the two edges leaving the component are incident to the same node, the third edge incident to that node would be a bridge. This is a contradiction, since $A$ is nonbridge. Thus, the two edges leaving the component must be incident to different nodes.

After we removed the pair of adjacent triangles, the two external nodes to which it leads have degree 2. To remedy this, we add a new node to the graph and join both nodes to it. Now the nodes are guaranteed to be connected again. We add another node and join the two new nodes together. Finally, we perform a cycle insertion with any arbitrary edge and the second new node. The resulting overall graph has $n+2$ nodes, and is connected, bridge, and cubic. No node's distance has increased after the construction.

\item[No Triangles] If there are no triangles, then we have two subcases. If there exists an edge that joins two nodes of the same distance from $v_0$, then we can do the following:

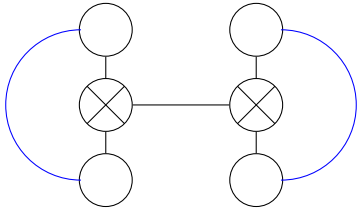
\begin{center}
\begin{tikzpicture}
[every node/.style={circle,draw=black,fill=white,inner sep=0,minimum size=20}]
\node (far1) at (1,3) {};
\node (mid1) at (1,2) {};
\node (close1) at (1,1) {};
\node (far2) at (3,3) {};
\node (mid2) at (3,2) {};
\node (close2) at (3,1) {};

\foreach \from/\to in {far1/mid1,mid1/close1,far2/mid2,mid2/close2,mid1/mid2}
\draw (\from) -- (\to);

\draw (2.75,2.25) -- (3.25,1.75);
\draw (2.75,1.75) -- (3.25,2.25);

\draw (0.75,2.25) -- (1.25,1.75);
\draw (0.75,1.75) -- (1.25,2.25);

\draw [blue] (0.67,3) arc [radius=1, start angle=90, end angle=270];
\draw [blue] (3.33,1) arc [radius=1, start angle=-90, end angle=90];
\end{tikzpicture}
\end{center}

The resulting graph is connected, bridge, and cubic. Since the removed edge did not facilitate any distance, it was not necessary for any path to $v_0$ for any node. Thus, no remaining node's distance has increased after the construction.

We discuss the second subcase later in the paper.\\
\end{description}

If we apply the two-node removal twice, we will have successfully removed four nodes. However, we have to account for the fact that the resulting graphs are not necessarily unique. First of all, the three cases could generate identical graphs, so we have to include a constant multiplier of 3x. Second, we cannot uniquely identify generating graphs from the post-construction resulting graphs. All we know is that the removed nodes were in the subgraph $A'$. We can use the size of $A'$ to bound the amount of possible duplication. Finally, in the `No Triangles' case, every pair of edges has to be considered when trying to recover the generating graph, as opposed to just creating a triangle out of every node as in the `Isolated Triangles' case.

\section{Analysis}

When we account for the loss of uniqueness that takes place during the removal stage, we end up with a gain on the order of $\sqrt[5]{k}$, where $k$ is the number of distinct edges in cycles in $A$. Equivalently, for each biconnected cubic graph of size $n$, we can generate an amount of cubic bridge graphs of size $n$ on the order of $\sqrt[5]{k}$. As $k$ grows larger, $\sqrt[5]{k}$ similarly grows larger without bound, so we are able to generate arbitrarily many graphs for sufficiently large $n$.

However, we have not yet accounted for the second subcase of the `No Triangles` case in the previous section. If the graph $A'$ has no triangles, then it is a complete tree. If we can sufficiently limit the size of the set of graphs in this subcase, then we can guarantee that we can generate arbitrarily many bridge graphs from biconnected graphs for large $n$.

We now pose a conjecture and state our final result:

\begin{conjecture}
	For large $n$, there exists an injection from the set of cubic bridge graphs of size $n+2$ with a complete tree of size in the range $\left[ \frac{\sqrt[5]{k}}{2}, \sqrt[5]{k} \right)$ rooted at the bridge to the set of cubic bridge graphs of size $n$.
\end{conjecture}

\begin{theorem}
	If Conjecture 6.1 holds, then almost all non-3-connected cubic non-Hamiltonian graphs are bridge graphs.
\end{theorem}

\bibliographystyle{amsplain}

\end{document}